\documentclass[11pt]{amsart}

\usepackage{amsmath}
\usepackage{amstext}
\usepackage{amssymb}
\usepackage{amsthm}
\usepackage{enumerate}

\makeatletter
\def\imod#1{\allowbreak\mkern10mu({\operator@font mod}\,\,#1)}
\makeatother

\newtheorem{theorem}{Theorem}[section]

\newtheorem{lemma}[theorem]{Lemma}
\newtheorem{claim}[theorem]{Claim}

\newtheorem*{thm}{Theorem}

\theoremstyle{definition}
\newtheorem{definition}[theorem]{Definition}

\newtheorem{corollary}[theorem]{Corollary}
\theoremstyle{remark}

\theoremstyle{remark}

\numberwithin{equation}{section}

    \DeclareMathOperator{\dom}{dom}
    
    \DeclareMathOperator{\ran}{ran}

    \DeclareMathOperator{\fix}{fix}

\newcommand{\rest}{{\upharpoonright}}

    \newcommand{\What}[1]{\widehat W_{\!\! #1}}

    \newcommand{\la}{\langle}
    \newcommand{\ra}{\rangle}

\def\Lim{{\hbox{Lim}}}
\def\fin{{\hbox{fin}}}
\def\M{{\mathcal M}}

\def\PP{{\mathbb P}}
\def\QQ{{\mathbb Q}}

\def\FF{{\mathbb F}}

\date{\today}
\begin{document}

\title[Definable maximal cofinitary groups]{Definable maximal cofinitary groups}

\author{Vera Fischer}

\address{Kurt G\"odel Research Center, University of Vienna, W\"ahringer Strasse 25, 1090 Vienna, Austria}
\email{vera.fischer@univie.ac.at}

\author{Sy David Friedman}
\address{Kurt G\"odel Research Center, University of Vienna, W\"ahringer Strasse 25, 1090 Vienna, Austria}
\email{sdf@logic.univie.ac.at}

\author{Asger T\"ornquist}

\address{Department of Mathematical Sciences, University of Copenhagen, Universitetspark 5, 2100 Copenhagen, Denmark}
\email{asgert@math.ku.dk}

\thanks{The first author would like to thank the Austrian Science Fund (FWF) for the generous support
through grant no. M1365-N13. The second author would also like to thank the Austrian Science Fund (FWF) for its generous support through
Project P 25748. The third author thanks the Danish Council for Independent Research for
generous support through grant no. 10-082689/FNU}




\begin{abstract} Using countable support iteration of $S$-proper posets, for some appropriate stationary set $S$, we obtain
a generic extension of the constructible universe, in which $\mathfrak{b}=\mathfrak{c}=\aleph_2$ and there is a maximal cofinitary group with a $\Pi^1_2$-definable set of generators.
\end{abstract}
\maketitle

\section{Introduction}

Following standard notation, we denote by $S_\infty$ the set of all permutations of the natural numbers. A function $f\in S_\infty$ is said to be {\emph{a cofinitary permutation}}, if it has only finitely many fixed points. A subgroup $\mathcal{G}$ of $S_\infty$ is said to be {\emph{a cofinitary group}}
if each of its non-identity elements has only finitely many fixed points, i.e. is a cofinitary permutation. A maximal cofinitary group, abbreviated mcg, is a cofinitary group, which is maximal with respect to these properties, under inclusion.  The minimal size of a maximal cofinitary group is denoted $\mathfrak{a}_g$. It is known that $\mathfrak{b}\leq\mathfrak{a}_g$ (see~\cite{BK}).

There has been significant interest towards the existence of maximal cofinitary groups which are low in the projective hierarchy.
The existence of a closed maximal cofinitary group is still open, while S. Gao and Y. Zhang (see~\cite{SGYZ}) showed that the axiom of constructibility implies the existence of a maximal cofinitary group with a co-analytic generating set. The result was improved by B. Kastermans, who showed that in the constructible universe $L$ there is a co-analytic maximal cofinitary group (see~\cite{BK}).

There is little known about the existence of nicely definable maximal cofinitary groups in models of $\mathfrak{c}>\aleph_1$. Our main result can be formulated as follows:

\begin{thm} There is a generic extension of the constructible universe in which $\mathfrak{b}=\mathfrak{c}=\aleph_2$ and there is a maximal cofinitary group with a $\Pi^1_2$-definable set of generators.
\end{thm}

The extension is obtained via a countable support iteration of $S$-proper posets, for some appropriate stationary set $S$. Along the iteration cofinally often we add generic permutations  which using a ground model set of almost disjoint functions
provide codes for themselves. Of use for this construction is on the one hand the poset for adding a maximal cofinitary group of desired cardinality, developed in~\cite{VFAT}, and on the other hand the coding techniques of~\cite{VFSF} and~\cite{SFLZ}.

The paper is organized as follows: in section 2 we give an outline of a poset which adjoins a cofinitary permutation to a given co-fnitary group and describe our main coding techniques; section 3 contains a detailed proof of our main theorem and in section 4 we conclude with the discussion of some remaining open questions.

\section{Maximal Cofinitary Groups and Coding}


\subsection{Adding generic permutations}

Our methods for adding a generic permutation are based on~\cite{VFAT}, where the first and third authors provide a poset which given an arbitrary index set $A$ and a (freely generated) cofinitary group $G$, generically adjoins a family of permutations $\{g_a\}_{a\in A}$  such that the group generated by $G\cup\{g_a\}_{a\in A}$ is cofinitary. We will be interested in the particular case in which $|A|=1$. Following the terminology of~\cite{VFAT}, given a non-empty set $B$, a mapping $\rho: B\to S_\infty$ is said to {\emph{induce a cofinitary representation}} if the natural extension of $\rho$ to a mapping $\hat{\rho}: \FF_B\to S_\infty$, where $\FF_B$ denotes the free group on the set $B$, has the property that its image is a cofinitary group. For $A\neq\emptyset$, we denote by $W_A$ the set of all reduced words on the alphabet $A$ and by $\What{A}$ the set of all words on the same alphabet which start and end with a different letter, or are a power of a single letter. We refer to the elements of $\What{A}$ as {\emph{good words}}. Note that every word is a conjugate of a good word, that is $\forall w\in W_A\exists w_0\in \What{A}\exists u\in W_A$ such that $w=uw_0u^{-1}$. The empty word is not a good word.

Whenever  $a$ is an index, which does not belong to the set $B$, $s$ is a finite partial injection from $\omega$ to $\omega$, $\rho:B\to S_\infty$ is a mapping which induces a cofinitary representation and $w$ is a reduced word on the alphabet $\{a\}\cup B$, we denote by $e_w[s,\rho]$ the (partial) function obtained by substituting every appearance of a letter $b$ from $B$ with $\rho(b)$, and every appearance of the letter $a$ with the partial mapping $s$.  By definition, let $e_\emptyset [s,\rho]$ be the identity. For the exact recursive definition see~\cite{VFAT}. Note that if $s$ is injective, then so is $e_w[s,\rho]$ (see~\cite{VFAT}).

\begin{definition} Let $B$ be a non-empty set, $a\notin B$ and $\rho:B\to S_\infty$ a mapping which induces a cofinitary representation. The poset $\QQ_{\{a\},\rho}$ consists of all pairs $(s,F)$ where $s\in{^{<\omega}\omega}$ is a finite partial injection, $F$ is a finite set of words in $\What{\{a\}\cup B}$. The extension relation states that $(t,H)\leq (s,F)$ if and only if $t$ end-extends $s$, $F\subseteq H$ and $\forall w\in F\forall n\in\omega$ if $e_w[t,\rho](n)=n$ then $e_w[s,\rho](n)$ is already defined (and so $e_w[s,\rho](n)=n$).
\end{definition}

Recall that a poset $\PP$ is said to be $\sigma$-centered, if $\PP=\bigcup_{n\in\omega}\PP_n$ where for each $n$, $\PP_n$ is centered, that is whenever $p,q$ are conditions in $\PP_n$ then there is $r\in\PP_n$ which is their common extension.  Note that $\QQ_{\{a\},\rho}$ is $\sigma$-centered. If $G$ is $\QQ_{\{a\},\rho}$-generic, then $g=\bigcup\{s:\exists F(s,F)\in G\}$ is a cofinitary permutation such that the mapping $\rho_G:\{a\}\cup B\to S_\infty$ defined by $\rho_G(a)=g$ and $\rho_G\rest B=\rho$, induces a cofinitary representation in $V[G]$. For the proofs of both of these statements see~\cite{VFAT}.

\subsection{Coding with a ground model almost disjoint family of functions}

We work over the constructible universe $L$. Recall that a $\hbox{ZF}^-$ model $M$ is said to be \emph{suitable} iff
$$M\vDash(\omega_2\;\hbox{exists and }\omega_2=\omega_2^L).$$
In our construction, we will use a family $\mathcal{F}=\{f_{\la \zeta,\xi\ra}:\zeta\in\omega\cdot 2,\xi\in {\omega_1}^L\}\in L$ of almost disjoint bijective functions such that $\mathcal{F}\cap M=\{ f_{\la \zeta,\xi\ra}:\zeta\in\omega\cdot 2, \xi\in (\omega_1^L)^M\}$ for every transitive model
$M$ of $ZF^-$ (see~\cite[Proposition 3]{SFLZ}).

For our purposes, we will need the following Lemma, which is analogous to~\cite[Proposition 4]{SFLZ}.

\begin{lemma} There is a sequence $\bar{S}=\la S_\beta:\beta<\omega_2\ra$ of almost disjoint stationary subsets of $\omega_1$, which is $\Sigma_1$ definable over $L_{\omega_2}$ with parameter $\omega_1$, and whenever $M,N$ are suitable models of $ZF^-$ such that $\omega_1^M=\omega_1^N$, then $\bar{S}^M$ agrees with $\bar{S}^N$ on $\omega_2^M\cap \omega_2^N$.
\end{lemma}
\begin{proof} Let $\la D_\gamma:\gamma <\omega_1\ra$ be the canonical $L_{\omega_1}$ definable $\Diamond$ sequence (see~\cite{Devlin}) and for each $\alpha<\omega_2$ let $A_\alpha$ be the $L$-least subset of $\omega_1$ coding $\alpha$. Now, let $S_\alpha:=\{i<\omega_1: D_i=A_\alpha\cap i\}$.
\end{proof}

Let $\bar{S}$ be as in the preceding Lemma and let $S$ be a stationary subset of $\omega_1$ which is almost disjoint from every element of $\bar{S}$.  We will use the following coding of an ordinal $\alpha<\omega_2$ by a subset of $\omega_1$ (see~\cite[Fact 5]{SFLZ}).

\begin{lemma} There is a formula $\phi(x,y)$
and for every $\alpha<\omega_2^L$ a set $X_\alpha\in ([\omega_1]^{\omega_1})^L$ such that
\begin{itemize}
 \item for every suitable model $M$ containing $X_\alpha\cap\omega_1^M$, $\phi(x, X_\alpha\cap\omega_1^M)$ has a unique solution in $M$, and this solution equals
$\alpha$ provided $\omega_1=\omega_1^M$.
 \item for arbitrary suitable models $M,N$ with $\omega_1^M=\omega_1^N$ and $X_\alpha\cap \omega_1^M\in M\cap N$, the solutions of $\phi(x, X_\alpha\cap \omega_1^M)$
in $M,N$ coincide.
\end{itemize}
\end{lemma}

\section{$\Pi^1_2$-definable set of generators}

In this section we will provide a generic extension of the constructible universe $L$ in which $\mathfrak{b}=\mathfrak{c}=\aleph_2$ and there is a maximal cofinitary group with a $\Pi^1_2$-definable set of generators. Fix a recursive bijection $\psi:\omega\times\omega\to\omega$.
Recursively define a countable support iteration of $S$-proper posets $\la \PP_\alpha,\dot{\QQ}_\beta:\alpha\leq\omega_2,\beta<\omega_2\ra$ as follows. If $\alpha <\omega_1$ let $\dot{\QQ}_\alpha$ be a $\PP_\alpha$-name for Hechler forcing for adding a dominating real.\footnote{For bookkeeping reasons it is more convenient to introduce the generators of the maximal cofinitary group at limit stages greater or equal $\omega_1$.}  Suppose $\PP_\alpha$ has been defined and
\begin{itemize}
\item for every $\beta\in \Lim(\alpha\backslash \omega_1)$ the poset $\QQ_\beta$ adds a cofinitary
permutation $g_\beta$, and
\item  the mapping $\rho_\beta: \Lim(\alpha\backslash\omega_1)\to S_\infty$  where $\rho_\alpha(\beta)=g_\beta$ induces a cofinitary representation.
\end{itemize}
In $L^{\PP_\alpha}$ define $\QQ_\alpha$ as follows. {\emph{If $\alpha$ is a successor}}, then $\QQ_\alpha$ is a $\PP_\alpha$-name for Hechler forcing for adding a dominating real. {\emph{If $\alpha\geq\omega_1$ is a limit}}, then $\alpha=\omega_1\cdot\nu+\omega\cdot\eta$ for some $\nu\neq 0$, $\nu <\omega_2$, $\eta<\omega_1$ and the conditions of $\QQ_\alpha$ are pairs $\la \la s, F,s^*\ra,\la c_k,y_k\ra_{k\in\omega}\ra$ where
\begin{enumerate}
\item $(s,F)\in \QQ_{\{\alpha\},\rho_\alpha}$;
\item $\forall k\in\omega$, $c_k$ is a closed bounded subset of $\omega_1\backslash \eta$ such that $c_k\cap S_{\alpha +k}=\emptyset$;
\item $\forall k\in\omega$, $y_k$ is a $0,1$-valued function whose domain $|y_k|$ is a countable limit ordinal, such that
$\eta\leq |y_k|$, $y_k\rest \eta=0$ and for every $\gamma$ such that $\eta\leq\gamma <|y_k|$, $y_k(2\gamma)=1$ if and only if
$\gamma\in\eta + X_\alpha=\{\eta+\mu:\mu\in X_\alpha\}$;
\item for every $k\in\psi[s]$ and every countable suitable model $M$ of $ZF^-$ such that $\xi=\omega_1^M\leq |y_k|$, $\xi$ is a limit point of $c_k$
and $y_k\rest \xi$, $c_k\cap \xi$ are elements of $M$, we have that
 $$M\vDash y_k\rest\xi\;\hbox{codes a limit ordinal}\;\bar{\alpha}\;\hbox{such that}\;S_{\bar{\alpha}+k}\;\hbox{is non-stationary}.$$
\item $s^*$ is a finite subset of $\{ f_{m,\xi}:m\in\psi[s], \xi\in c_m\}\cup\{ f_{\omega +m,\xi}: m\in\psi[s], y_m(\xi)=1\}$.
\end{enumerate}
The extension relation states that  $\bar{q}=\la\la t,H,t^*\ra, \la d_k,z_k\ra_{k\in\omega}\ra$ extends $\bar{p}=\la\la s,F,s^*\ra, \la c_k,y_k \ra_{k\in\omega}\ra$ iff
\begin{enumerate}
\item $(t,H)\leq_{\QQ_{\{\alpha\},\rho_\alpha}} (s,F)$,
\item $\forall f\in s^*$, $t\backslash s\cap f=\emptyset$,
\item $\forall k\in \psi[s]$, $d_k$ end-extends $c_k$ and $y_k\subseteq z_k$
\end{enumerate}

With this the recursive definition of $\PP_{\omega_2}$ is complete. If $\bar{p}\in\QQ_\alpha$, where $\bar{p}=\la\la s,F,s^*\ra,\la c_k,y_k\ra_{k\in\omega}\ra$
we write $\fin(\bar{p})$ for $\la s,F,s^*\ra$ and $\inf(\bar{p})$ for $\la c_k, y_k\ra_{k\in\omega}$. In particular $\fin(\bar{p})_0=s$.

\begin{lemma}\label{domain0} For every condition $\bar{p}=\la\la s, F,s^*\ra,\la c_k,y_k\ra_{k\in\omega}\ra\in\QQ_\alpha$ and every $\gamma\in\omega_1$ there exists a sequence $\la d_k, z_k\ra_{k\in\omega}$ such that $\bar{q}=\la\la s,F,s^*\ra,\la d_k, z_k\ra_{k\in\omega}\ra\in \QQ_\alpha$,
$\bar{q}\leq\bar{p}$ and for all $k\in\omega$ we have that $|z_k|,\max d_k\geq \gamma$.
\end{lemma}
\begin{proof} As in~\cite[Lemma 1.1]{VFSF}.
\end{proof}

\begin{lemma}\label{preprocessed0} For every $p\in\QQ_\alpha$ and every dense open set $D\subseteq\QQ_\alpha$, there is $q\leq p$ such that $\fin(q)=\fin(p)$ and for every $p_1\in D$, $p_1\leq q$ there is $p_2\in D$, $p_2\leq q$ such that $\fin(p_2)_0=\fin(p_1)_0$ and $\inf(p_2)=\inf(q)$.
\end{lemma}
\begin{proof} Let $p=\la \la t_0, F_0, t_0^*\ra,\la d^0_k,z^0_k\ra_{k\in\omega}\ra$. Let $\M$ be a countable elementary submodel of $L_\Theta$, for $\Theta$ a sufficiently large regular cardinal, which contains $\QQ_\alpha$, $\bar{p}$, $X_\alpha$, $D$ as elements and such that $j=\mathcal{M}\cap\omega_1\notin \bigcup_{k\in\psi[t_0]} S_{\alpha+k}$.
Let $\la \bar{r}_n, s_n\ra_{n\in\omega}$ enumerate all pairs $\la \bar{r}_n, s_n\ra$ where $\bar{r}_n\in \QQ_\alpha\cap\M$, $s_n$ is a finite partial injective function from $\omega$ to $\omega$ and each pair is enumerated cofinally often. Let $\{j_n\}_{n\in\omega}$ be an increasing sequence which is cofinal in $j$. Inductively we will construct a decreasing sequence $\la \bar{p}_n\ra_{n\in\omega}\subseteq\QQ\cap\M$ such that for all $n$, $\fin(\bar{p}_n)=\fin(\bar{p})$.

Let $\bar{p}_0=\bar{p}$. Suppose $\bar{p}_n$ has been defined. If there is $\bar{r}_{1,n}\in\M\cap\QQ$ such that $\bar{r}_{1,n}\leq \bar{p}_n,\bar{r}_n$ and $\fin(\bar{r}_{1,n})=s_n$ then extend $\inf(\bar{r}_{1,n})$ to a sequence $\la d^{n+1}_k,z^{n+1}_k\ra_{k\in\omega}$ in $\M$ in such a way that for all $k\in\omega$, $\max d^{n+1}_k\geq j_n$, $|z^{n+1}_k|\geq j_n$. Then let $\bar{p}_{n+1}= \la \fin(\bar{p}_0),\la d^{n+1}_k,z^{n+1}_k\ra_{k\in\omega}\ra$. If there is no such $\bar{r}_{1,n}$, then extend $\inf(\bar{p}_n)$ to a sequence  $\la d^{n+1}_k,z^{n+1}_k\ra_{k\in\omega}$ in $\M$ such that for all $k\in\omega$, $\max d^{n+1}_k\geq j_n$, $|z^{n+1}_k|\geq j_n$. With this the inductive construction is complete. For every $k\in\omega$, let $d_k=\bigcup_{n\in\omega} d^n_k\cup\{j\}$ and $z_k=\bigcup_{n\in\omega} z^n_k$. Let $q=\la \fin(\bar{p}), \la d_k,z_k\ra_{k\in\omega}\ra$.

We will show that $q$ is indeed a condition. For this we only need to verify part $(4)$ of being a condition, since the other clauses are clear. Fix $k\in\psi[t_0]$. Let $\M_0$ be a countable suitable model of $ZF^-$ such that $\omega_1^{\M_0}=j$ and $z_k,d_k$ are elements of $\M_0$. Let $\bar{\M}$ be the Mostowski collapse of the model $\M$ and let $\pi:\M\to\bar{\M}$ be the corresponding isomorphism. Note that $j=\omega_1\cap\M=\omega_1^{\bar{\M}}$. Since $X_\alpha\in\M$ and $\M$ is an elementary submodel of $L_\Theta$, $\alpha$ is the unique solution of $\phi(x, X_\alpha)$ in $\M$. Therefore $\bar{\alpha}=\pi(\alpha)$ is the unique solution of $\phi(x, X_\alpha\cap j)=\phi(x, \pi(X_\alpha))$ in $\bar{\M}$. Note also that $S^{\bar{\M}}_{\bar{\alpha}+k}=\pi{(S_{\alpha+k})}=S_{\alpha+k}\cap j$. Since $\omega_1^{\bar{\M}}=\omega_1^{\M_0}$ and $X_\alpha\cap j\in\bar{\M}\cap\M_0$, the solutions of $\phi(x, X_\alpha\cap j)$ in $\bar{\M}$ and $\M_0$ coincide. That is, the solution of $\phi(x, X_\alpha\cap j)$ in $\M_0$ is $\bar{\alpha}$. By the properties of the sequence of stationary sets which we fixed in the ground model, we have $S^{\M_0}_{\bar{\alpha}+k}=S^{\bar{\M}}_{\bar{\alpha}+k}=\pi(S^\M_{\alpha +k})=S_{\alpha+k}\cap j$. Since $d_k\in \M_0$ and $d_k$ is unbounded in $j$, we obtain that $S^{\M_0}_{\bar{\alpha}+k}$ is not stationary in $\M_0$. Therefore $q$ is indeed a condition.

Consider an arbitrary extension $p_1=\la\fin(p_1),\inf(p_1)\ra$ of $\bar{q}$ from the dense open set $D$ and let $\fin(p_1)_0=r_1$. Then $\la r_1, F_0, t_0^*\ra\in\M$, and so for some $m$, $\bar{r}^*=\la\la r_1, F_0, t_0^*\ra, \la d^m_k,z^m_k\ra_{k\in\omega}\ra\in\QQ_\alpha\cap \M$.
Then there is some $n\geq m$ such that $s_n=r_1$, $\bar{r}_n=\bar{r}^*$. Note that $p_1\leq q,\bar{r}_n$ and so $p_1$ is a common extension of $\bar{p}_n$, $\bar{r}_n$.  By elementarity there is $\bar{r}_{1,n}\in\M\cap D$
which is a common extension of $\bar{p}_n$, $\bar{r}_n$, such that $\fin(\bar{r}_{1,n})=\la r_1=s_n, F_2, r_2^*\ra$. Let $p_2:=\la \la r_1, F_2, r_2^*\ra, \la d_k,z_k\ra_{k\in\omega}\ra$. Note that $\inf(\bar{p}_{n+1})$ extends $\inf(\bar{r}_{1,n})$ and so $p_2\leq \bar{r}_{1,n}$, which implies that $p_2\in D$.  Clearly $p_2\leq q$ and so $p_2$ is as desired.
\end{proof}

\begin{lemma}\label{generic0} Let $\M$ be a countable elementary submodel of $L_\Theta$ for sufficiently large $\Theta$ containing all relevant parameters,
$i=\M\cap \omega_1$, $\bar{p}=\la\la s,F,s^*\ra, \la d^0_k, z^0_k\ra_{k\in\omega}\ra$ an element of $\M\cap\QQ_\alpha$. If $i\notin\bigcup_{k\in \psi[s]} S_{\alpha+k}$, then there exists an $(\M, \QQ_\alpha)$-generic condition $\bar{q}\leq\bar{p}$ such that $\fin(\bar{q})=\fin(\bar{p})$.
\end{lemma}
\begin{proof} Let $\{ D_n\}_{n\in\omega}$ be an enumeration of
all dense open subsets of $\QQ_\alpha$ from $\M$ and let $\{i_n\}_{n\in\omega}$ be an increasing sequence which is cofinal in $i$. Inductively, construct a sequence $\la \bar{q}_n\ra_{n\in\omega}\subseteq \M\cap\QQ_\alpha$ such that $\bar{q}_0=\bar{p}$, and
\begin{enumerate}
\item for every $n\in\omega$, $\bar{q}_{n+1}\leq\bar{q}_n$, $\fin(q_n)=\fin(\bar{p})$;
\item if $\inf(q_n)=\la d^n_k, z^n_k\ra_{k\in\omega}$  then for all $k\in\omega$, $\max d^n_k\geq i_n$, $|z^n_k|\geq i_n$;
\item for every $\bar{p}_1\in D_n$ extending $\bar{q}_n$, there is $\bar{p}_2\in D_n$ which extends $\bar{q}_n$ and such that $\fin(\bar{p}_2)_0=\fin(\bar{p}_1)_0$, $ \inf(\bar{p}_2)=\inf(\bar{q}_n)$.
\end{enumerate}

Now define a condition $\bar{q}$ such that $\fin(\bar{q})=\fin(\bar{p})$, $\inf(\bar{q})=\la d_k,z_k\ra_{k\in\omega}$ where $d_k=\bigcup_{n\in\omega} d^n_k\cup\{i\}$, $z_k=\bigcup_{n\in\omega} z^n_k$. To verify that $\bar{q}$ is indeed a condition, proceed as in the proof of $q$ being a condition from Lemma~\ref{preprocessed0}. Then $\bar{q}\leq\bar{p}$ and we will show that $\bar{q}$ is $(\M,\QQ_\alpha)$-generic. For this it is sufficient to show that for every $n\in\omega$, the set $D_n\cap\M$ is predense below $\bar{q}$. Thus fix some $n\in\omega$ and $\bar{p}_1=\la\la t_1, F_1, t^*_1\ra,\inf(\bar{p}_1)\ra$ an arbitrary extension of $\bar{q}$. Without loss of generality $\bar{p}_1\in D_n$. Since $\bar{p}_1\leq\bar{q}_n$
we obtain the existence of $F_2, t_2^*\in\M$ such that $\bar{p}_2=\la \la t_1, F_2, t^*_2\ra,\la d^n_k, z^n_k\ra_{k\in\omega}\ra\leq\bar{q}_n$ and $\bar{p}_2\in\M\cap D_n$. Then $\bar{p}_3=\la \la t_1, F_1\cup F_2, t^*_1\cup t^*_2\ra, \inf(\bar{p}_1)\ra$ is a common extension of $\bar{p}_1$ and $\bar{p}_2$.
\end{proof}

\begin{corollary}\label{nonkill0} For every $\alpha <\omega_2$, the poset $\QQ_\alpha$ is $S$-proper. Consequently, $\PP_{\omega_2}$ is $S$-proper and hence preserves cardinals. More precisely, for every condition $\bar{p}=\la\la s, F, s^*\ra,\la c_k,y_k\ra_{k\in\omega}\ra\ra\in\QQ^1_\alpha$ the poset $\{\bar{r}\in\QQ_\alpha:\bar{r}\leq\bar{p}\}$ is $\omega_1\backslash \bigcup_{n\in\psi[s]} S_{\alpha+ n}$-proper.
\end{corollary}

\subsection{Properties of $\QQ=\QQ_\alpha$}

Throughout the subsection, let $\alpha$ be a limit ordinal such that $\omega_1\leq\alpha<\omega_2$. We study the properties of $\QQ:=\QQ_\alpha$ in $L^{\PP_\alpha}$.

\begin{claim}[Domain Extension]\label{domainext0} For every condition $\bar{p}=\la \la s, F,s^*\ra,\la c_m,y_m\ra_{m\in\omega}\ra$, natural number $n$ such that $n\notin \dom(s)$ there are co-finitely many $m\in \omega$ such that $\la\la s\cup\{(n,m)\}, F,s^*\ra, \la c_m,y_m\ra_{m\in\omega}\ra$ is a condition extending $\bar{p}$.
\end{claim}
\begin{proof} Fix $\bar{p}$, $n$ as above. By~\cite[Lemma 2.7]{VFAT} there is a co-finite set $I$ such that for all $m\in I$ $(s\cup\{(n,m)\},F)\leq_{\QQ_{\{\alpha\}, \rho_\alpha}} (s,F)$. Since $s^*$ is finite, we can define $N_0=\max\{f(n):n\in s^*\}$. Then for every $m\in I\backslash N_0$, $$\la \la s\cup\{(n,m)\}, F, s^*\ra, \la c_k, y_k\ra_{k\in\omega}\ra\leq\bar{p}.$$
\end{proof}

\begin{claim}[Range Extension]\label{rangeext0} For any condition $\bar{p}=\la \la s,F,s^*\ra,\la c_m,y_m\ra_{m\in\omega}\ra$, natural number $m\notin \ran(s)$ there are co-finitely many $n\in\omega$ such that $\la \la s\cup \{(n,m)\}, F,s^*\ra,\la c_k,y_k\ra_{k\in\omega}\ra$ is a condition, extending $\bar{p}$.
\end{claim}
\begin{proof} Fix $\bar{p}$, $m$ as above. By~\cite[Lemma 2.7]{VFAT} there is a co-finite set $I$ such that for all $n\in I$, $(s\cup\{(n,m)\}, F)\leq_{\QQ_{\{\alpha\},\rho_\alpha}} (s,F)$. Now for every $n$, consider the set $A_n=\{f(n)\}_{f\in s^*}$. If there are infinitely many $n$ such that $m\in A_n$ then $\exists f\in s^*\exists^\infty n$ such that $f(n)=m$, which is a contradiction to $f$ being a bijection. That is $\forall^\infty n(m\notin A_n)$. Choose $N$ such that $\forall n\geq N(m\notin A_n)$. Then $\forall n\in I\backslash N(\la\la s\cup\{(n,m)\}, F, s^*\ra,\la c_k, y_k\ra_{k\in\omega}\ra)$ is an extension of $\bar{p}$ with the desired properties.
\end{proof}

The following claim is straightforward.

\begin{claim} For every $w_0\in\What{\{\alpha\}\cup\Lim(\alpha\backslash\omega_1)}$  the set $D_{w_0}=\{\bar{p}\in\QQ: w_0\in \hbox{fin}(\bar{p})_1\}$ is dense.
\end{claim}

\begin{claim} Suppose $\bar{q}=\la\la s,F, s^*\ra,\la c_k, y_k\ra_{k\in\omega}\ra\Vdash_{\QQ_\alpha} e_w[\rho_G](n)=n$ for some  $w\in \What{\{\alpha\}\cup\Lim(\alpha\backslash\omega_1)}$. Then $e_w[s,\rho_\alpha](n)$ is defined and $e_w[s,\rho_\alpha](n)=n$.
\end{claim}
\begin{proof} Let $G$ be $\QQ_\alpha$ generic over $L^{\PP_\alpha}$ such that $q\in G$. By definition of the extension relation there is a condition
$\bar{r}=\la\la t,H,t^*\ra,\la d_k,z_k\ra_{k\in\omega}\ra$  in $G$ such that $e_w[\rho_G](n)=e_w[t,\rho_\alpha](n)=n$. Then $(t,H)\leq_{\QQ_{\{\alpha\},\rho_\alpha}} (s,F)$ and since the extension of $\QQ_{\{\alpha\},\rho_\alpha}$ does not allow new fixed points we obtain
$e_w[s,\rho_\alpha](n)=n$.
\end{proof}

\begin{lemma} Let $G$ be $\QQ_\alpha$-generic over $L^{\PP_\alpha}$ and let $g_\alpha=\bigcup_{\bar{p}\in G} \fin(\bar{p})_0$. Then $g_\alpha$ is a cofinitary permutation and $\la g_\beta\ra_{\beta\leq \alpha}$ is a cofinitary group.
\end{lemma}
\begin{proof} Since for every $n,m$ in $\omega$, the sets $D_n=\{\bar{p}\in\QQ:n\in\dom(\fin(\bar{p})_0)\}$, $R_m=\{\bar{p}\in\QQ, m\in\ran(\fin(\bar{p})_0)\}$ are dense, it is easy to see than $g=g_\alpha$ is a surjective function. Injectivity follows directly from the properties of $\QQ_{\{\alpha\},\rho_\alpha}$ (see~\cite{VFAT}), and so $g$ is a permutation.

We will show that the group generated by $\{g_\beta\}_{\beta\in \Lim(\alpha\backslash \omega_1)}\cup\{g_\alpha\}$ is a cofinitary group. Fix an arbitrary word $w\in W_{\{\alpha\}\cup\Lim(\alpha\backslash \omega_1)}$. Then there are $w^\prime\in \What{\{\alpha\}\cup\Lim(\alpha\backslash \omega_1)}$ and $u\in W_{\{\alpha\}\cup\Lim(\alpha\backslash \omega_1)}$ such that $w=u^{-1}w^\prime u$. Since $D_{w^\prime}$ is dense, there is a condition $\bar{p}=\la\la s,F,s^*\ra,\la c_k,y_k\ra_{k\in\omega}\ra$ in $G$  such that $w^\prime\in F$. Suppose $e_{w^\prime}[\rho_G](n)=n$. Then there is $\bar{q}\in G$, $\bar{q}\leq\bar{p}$ such that  $\bar{q}\Vdash e_{w^\prime}[\rho_G](n)=n$. By the above Lemma,
$e_{w^\prime}[t,\rho_\alpha](n)=n$, where $\bar{q}=\la (t, F^\prime, t^*), \la d_k,z_k\ra_{k\in\omega}\ra$ and so by the extension relation $e_{w^\prime}[s,\rho_\alpha](n)=n$. Then $\fix(e_{w^\prime}[\rho_G])=\fix (e_{w^\prime}[s,\rho_\alpha])$ which is finite and so $\fix(e_w[\rho_G])$ is also finite.
\end{proof}

\begin{lemma}[Generic Hitting] In $L^{\PP_\alpha}$ suppose $\la \{h\}\cup\{g_\beta\}_{\beta<\alpha}\ra$ is a cofinitary group and $h$ is not covered by finitely many members of $\mathcal{F}$ with indices above $\eta$. Then $L^{\PP_{\alpha+1}}\vDash \exists^\infty n\in\omega(g_\alpha(n)=h(n))$.
\end{lemma}
\begin{proof} We claim that for every $N\in\omega$, the set $D_N=\{\bar{q}:\exists n\geq N(s(n)=h(n))\}$ is dense in $\QQ_\alpha$. Let $\bar{p}=\la\la s,F,s^*\ra,\la c_k,y_k\ra_{k\in\omega}\ra$ be an arbitrary condition. By~\cite[Lemma 2.19]{VFAT} there is $N$ such that for all $n\geq N$,
$$(s\cup\{(n, h(n))\},F)\leq_{\QQ_{\{\alpha\},\rho_\alpha}} (s,F).$$
Since $h$ is not covered by the members of $s^*$, we have that  $\exists^\infty n$ such that $h(n)\notin \{f(n)\}_{f\in s^*}$.
Denote this set $I_h(\bar{p})$.  Let $n\in I_h(\bar{p})\backslash \max\{ N_{\bar{p}}, N\}$. Then
$$\bar{q}:=\la \la s\cup\{(n, h(n))\}, F, s^*\ra,\la c_k,y_k\ra_{k\in\omega}\ra\leq \bar{p}$$
and $\bar{q}\in D_N$. Therefore $L^{\PP_{\alpha+1}}\vDash \exists^\infty n (g_\alpha(n)=h(n))$.
\end{proof}

\begin{lemma} The group $\mathcal{G}:=\la g_\alpha\ra_{\alpha\in\Lim(\omega_2\backslash\omega_1)}$ added by $\PP_{\omega_2}$ is a maximal cofinitary group.
\end{lemma}
\begin{proof} Suppose $\mathcal{G}$ is not maximal. Then there is a cofinitary permutation $h$ such that
$$\la\{g_\alpha\}_{\alpha\in\Lim(\omega_2\backslash\omega_1)}\cup \{h\}\ra$$ is cofinitary. Let $\alpha<\omega_2$ be the least limit ordinal such that $\alpha=\omega_1\cdot\xi$ for some $\xi\neq 0$ and such that $h\in L^{\PP_\alpha}$. Then there is $\eta\geq 0$ such that $h$ is not covered by finitely many members of $\mathcal{F}$ whose second index is above $\eta$. Therefore by the Generic Hitting Lemma the poset $\QQ_{\omega_1\cdot\xi+\omega\cdot\eta}$ adds a generic permutation $g_{\omega_1\cdot\xi+\omega\cdot\eta}$ which is infinitely often equal to $h$, which is a contradiction.
\end{proof}

\subsection{Coding}

Let $G_\alpha$ be $\QQ_\alpha$-generic filter over $L^{\PP_{\alpha}}$ and let $g_\alpha:=\bigcup_{\bar{p}\in G} \hbox{fin}(\bar{p})_0$. For every $k\in\psi[g_\alpha]$ define $Y^\alpha_k:=\bigcup_{\bar{p}\in G_\alpha}\hbox{inf}(\bar{p})_1$,
$C^\alpha_k:=\bigcup_{\bar{p}\in G_\alpha} \hbox{inf}(\bar{p})_0$ and  $S^*:=\bigcup_{\bar{p}\in G_\alpha} \hbox{fin}(\bar{p})_2$. Let $G:=G_{\omega_2}$.

The following is clear using easy extendibility arguments together with Lemmas~\ref{domain0},~\ref{domainext0},~\ref{rangeext0}.

\begin{lemma} The sets $Y^\alpha_k$, $C^\alpha_k$, and $S^*$ have the following properties:
 \begin{itemize}
  \item $S^*=\{f_{\la m,\xi\ra}: m\in\psi[g_\alpha], \xi\in C^\alpha_m\}\cup\{ f_{\la\omega+m,\xi\ra}: m\in \psi[g_\alpha], Y^\alpha_m(\xi)=1\}$.
  \item If  $m\in\psi[g_\alpha]$ then $\dom(Y^\alpha_m)=\omega_1$ and $C^\alpha_m$ is a club in $\omega_1$ disjoint from $S_{\alpha+m}$.
  \item If $m\in\psi[g_\alpha]$ then $|g_\alpha\cap f_{\la m,\xi\ra}|<\omega$ if and only if $\xi\in C^\alpha_m$.
  \item If $m\in \psi[g_\alpha]$ then $|g_\alpha\cap f_{\la \omega+m,\xi\ra}|<\omega$ if and only if $Y^\alpha_m(\xi)=1$.
 \end{itemize}
\end{lemma}

\begin{corollary} Let $n\in \omega\backslash \psi[g_\alpha]$. Then $S_{\alpha+n}$ remains stationary in $L^{\PP_{\omega_2}}$.
\end{corollary}
\begin{proof} Let $G$ be $\PP_{\omega_2}$-generic over $L$ and let $p\in G$ such that $p\Vdash \beta\notin\{\alpha+n:n\in\psi[g_\alpha]\}$. Then $G$ is also $\PP_{\omega_2}(p)$-generic, where $\PP_{\omega_2}(p):=\{q:q\leq p\}$ is the countable support iteration of $\QQ_\gamma(p(\gamma))$ for $\gamma<\omega_2$.  However for every $\gamma$, the poset $\QQ_\gamma(p(\gamma))$ is $S_\beta$-proper and so the entire iteration is $S_\beta$-proper.
\end{proof}

\begin{lemma}\label{pi12definition} In $L[G]$ let $A=\{g_\alpha:\omega_1\leq\alpha<\omega_2,\alpha\;\hbox{limit}\}$. Then $g\in A$ if and only if for every countable suitable model $M$ of $ZF^-$ containing $g$ as an element there exists a limit ordinal $\bar{\alpha}<\omega_2^M$ such that $S^M_{\bar{\alpha}+k}$ is non-stationary in $M$ for all $k\in\psi[g]$.
\end{lemma}
\begin{proof} The proof is analogous to that of~\cite[Lemma 13]{SFLZ}.
Let $g\in A$. Find $\alpha <\omega_2$ such that $g=g_\alpha$, and let $M$ be a countable suitable model containing $g$ as an element. Then
$C^\alpha_k\cap\omega_1^M$, $Y^\alpha_k\rest \omega_1^M$ are elements of $M$ for all $k\in \psi[g_\alpha]$.  Fix any $m\in\psi[g_\alpha]$. Then there is $\bar{p}=\la \la s, F,s^*\ra,\la c_k, y_k\ra_{k\in\omega}\ra\in G$ such that
$m\in\psi[s]$ and $C_m^\alpha\cap \omega_1^\M = c_m$, $Y_m^\alpha\cap\omega_1^\M = y_m$. By definition of being a condition we obtain that
$$\M\vDash Y_m^\alpha\cap \omega_1^\M\;\hbox{codes a limit ordinal }\bar{\alpha}_m\;\hbox{such that}\; S_{\bar{\alpha}_m+m}\;\hbox{is not stationary}.$$
Note that for every distinct $m_1, m_2$ in $\psi[g_\alpha]$ we have that $Y_{m_1}^\alpha\cap\omega_1^\M = Y_{m_2}^\alpha\cap\omega_1^\M$, and so $\bar{\alpha}_m$ does not depend on $m$.

To see the other implication, fix $g$ such that for every countable suitable model containing $g$ as an element there
exists $\bar{\alpha}<\omega_2^M$ such that $S^M_{\bar{\alpha}+k}$ is non-stationary in $M$ for all $k\in\psi[g]$. By the L\"owenheim-Skolem theorem
the same holds for arbitrary suitable models of $ZF^-$ containing $g$. In particular this holds in $M=L_\Theta[G]$ for some sufficiently large $\Theta$,
say $\Theta>\omega_{100}$. Then $\omega_2^M=\omega_2^{L[G]}=\omega_2^{L}$, $\bar{S}^M=\bar{S}$, and the notions of stationarity of subsets of $\omega_1$ coincide in $M$ and $L[G]$. Thus there is a limit ordinal $\alpha<\omega_2$ such that $S_{\alpha+k}$ is non-stationary for all $k\in\psi[g]$. By the above corollary for every $\beta\notin \{\alpha+ k:k\in\psi[g_\alpha]\}$ the set $S_\beta$ is stationary. Therefore $\psi[g]\subseteq\psi[g_\alpha]$ and so $g=g_\alpha$.
\end{proof}

Thus as the right-hand side of the equivalence stated in Lemma~\ref{pi12definition} is $\Pi^1_2$, we obtain:

\begin{theorem} There is a generic extension of the constructible universe in which $\mathfrak{b}=\mathfrak{c}=\aleph_2$ and there is a maximal cofinitary group with a $\Pi^1_2$-definable set of generators.
\end{theorem}

\section{Remarks}

We expect that the techniques of~\cite{VFSFLZ} can be modified to produce a generic extension of the constructible universe in which $\mathfrak{b}=\mathfrak{c}=\aleph_3$ and there is a maximal cofinitary group with a $\Pi^1_2$-definable set of generators. Of interest remains the following question: Is it consistent that there is a $\Pi^1_2$ definable maximal cofinitary group and $\mathfrak{b}=\mathfrak{c}=\aleph_2$?

\end{document}